\documentclass[smallextended]{article}       
\usepackage{graphicx}
\DeclareGraphicsExtensions{.bmp} 
\usepackage{subfigure}
\usepackage{epstopdf}
\usepackage{cite}
\usepackage[centertags]{amsmath}
\usepackage{amsfonts,amssymb,newlfont,makeidx,listings,xcolor}
\usepackage[colorlinks=true,urlcolor=blue,linkcolor=blue,pdfborder={0 0 0}]{hyperref}
\usepackage[left=25mm,top=20mm,right=25mm,bottom=20mm]{geometry}

\newtheorem{theorem}{Theorem}[section]
\newtheorem{corollary}{Corollary}[section]
\newtheorem{remark}{Remark}[section]

\newenvironment{proof}{\textit{Proof}:}{\hfill$\square$}
\numberwithin{equation}{section}

\begin{document}

 \title{Quarter-symmetric non-metric connection}
\date{}
\author{\textbf{Miroslav D. Maksimovi\'c}}
\maketitle

	\noindent{University of Pri\v stina in Kosovska Mitrovica, Faculty of Sciences  and Mathematics, Department of Mathematics, Kosovska Mitrovica, Serbia,} mail: miroslav.maksimovic@pr.ac.rs

\begin{abstract}\small 
The paper will study a new quarter-symmetric non-metric connection on a generalized Riemannian manifold. It will determine the relations that the torsion tensor satisfies. The exterior derivative of the skew-symmetric part $F$ of basic tensor $G$ with respect to the Levi-Civita connection coincides with that of skew-symmetric part $F$ with respect to quarter-symmetric non-metric connection, which implies that the even-dimensional manifold endowed with non-degenerate 2-form $F$ is symplectic manifold if and only if it is closed with respect to quarter-symmetric non-metric connection. The linearly independent curvature tensors of this connection and its dual connection are determined and the properties of these tensors are discussed. Finally, the condition is given that the connection should be dual symmetric.


\end{abstract}

\section{Introduction and motivation}
The theory of linear connections holds an important place in Differential Geometry. In 1924, A. Friedmann and J. A. Schouten introduced the concept of semi-symmetric connection \cite{friedmannschouten1924}, while in 1932, H. A. Hayden defined metric connections with torsion tensor \cite{hayden1932}. E. Pak studied semi-symmetric connection on the pseudo-Riemannian manifold and given the relationship between semi-symmetric connection and Levi-Civita connection \cite{pak1969}. In paper \cite{Yano1970}, K. Yano studied semi-symmetric metric connection on the Riemannian manifold in detail. Afterward, many authors studied the application of this connection to other manifolds. With the aim of generalizing semi-symmetric connection, S. Golab introduced the notion of quarter-symmetric connection \cite{golab1975}. In \cite{mishrapandey1980}, R. S. Mishra and S. Pandey studied several types of quarter-symmetric connection on different manifolds, while in \cite{yano1982} K. Yano determined the relationship between quarter-symmetric metric connection and Levi-Civita connection.

In \cite{agashe1992}, N. Agashe and M. Chafle defined a semi-symmetric non-metric connection on a Riemannian manifold and observed properties of the curvature tensor and the Weyl projective curvature tensor. In papers \cite{chatpandey2015, dehanzhao2016}, the authors defined a new type of semi-symmetric non-metric connection given by the equation
\begin{equation*}
	\overset{1}{\nabla}_{X} Y = \overset{g}{\nabla}_{X} Y + a\pi(Y) X + b\pi(X) Y.
\end{equation*}
This connection was also studied in papers  \cite{chayil2019,cui2019}. These papers motivated us to study the quarter-symmetric connection on the generalized Riemannian manifold in the following form
\begin{equation}\label{eq:Q-Sconnection1}
	\overset{1}{\nabla}_{X} Y = \overset{g}{\nabla}_{X} Y + a\pi(Y) A X + b\pi(X) A Y, 
\end{equation}
whose torsion tensor given with
\begin{equation*}
	\overset{1}{T}(X,Y)= (a-b)(\pi(Y) A X - \pi(X) A Y),
\end{equation*}
where $a$ and $b$ are different real numbers (i.e. $a,b \in \mathbb{R}$, $a\ne b$) and $A$ is the (1,1) tensor associated with skew-symmetric tensor $F$. 

\begin{remark}
	For $a=0$ and $b\ne 0$, the quarter-symmetric connection (\ref{eq:Q-Sconnection1}) is metric (for example, in papers \cite{mishrapandey1980,yano1982,zlma2022} the case for $b=-1$ studied). For $a \ne 0$, $b=0$, the connection (\ref{eq:Q-Sconnection1}) is non-metric (for example, see \cite{bhowmik2010, dubey2010} or equations below equation (2.8) in \cite{zlma2022}). In this paper, we will study the connection (\ref{eq:Q-Sconnection1}), where neither $a$ nor $b$ are equal to zero.
\end{remark}

\section{Preliminaries}

Let $\mathcal{M}$ be an $n$-dimensional differentiable manifold and $G$ be a non-symmetric (0,2) tensor. The pair $(\mathcal{M}, G)$ is called the \textit{generalized Riemannian manifold}. Based on the non-symmetry, the \textit{basic tensor} $G$ can be represented as follows
\begin{equation*}
	G(X,Y) = g(X,Y) + F(X,Y),
\end{equation*}
where $g$ is its symmetric part and $F$ its skew-symmetric part, that is
\begin{equation*}
	g(X,Y) = \frac{1}{2} (G(X,Y) + G(Y,X) ), \quad F(X,Y) = \frac{1}{2} (G(X,Y) - G(Y,X)).
\end{equation*}
We assume that the symmetric part $g$ is non-degenerate of arbitrary signature. Symmetric part $g$ and skew-symmetric part $F$ are related as follows 
\begin{equation}\label{eq:FgA}
	F(X,Y) = g(AX,Y),
\end{equation}
where $A$ is the (1,1) tensor field associated with tensor $F$. Depending on the properties of tensor $A$, we can observe various examples of generalized Riemannian manifold (see \cite{ivanov2016}). 

In this paper, we will study the non-symmetric linear connection $\overset{1}{\nabla}$ on a generalized Riemannian manifold, i.e. linear connection $\overset{1}{\nabla}$ with torsion tensor
\begin{equation*}
	\overset{1}{T}(X,Y) = \overset{1}{\nabla}_{X} Y - \overset{1}{\nabla}_{Y} X - [X,Y],
\end{equation*}
where $[\cdot,\cdot]$ denote the Lie bracket and $X,Y$ are smooth vectors on a differentiable manifold $\mathcal{M}$. In the following, we will use the (0,3) torsion tensor defined with equation
\begin{equation*}
	\overset{1}{T}(X,Y,Z) = g(\overset{1}{T}(X,Y),Z).
\end{equation*}
The \textit{dual} connection $\overset{2}{\nabla}$ of connection $\overset{1}{\nabla}$ is defined with equation
\begin{equation*}
	\overset{2}{\nabla}_{X} Y = \overset{1}{\nabla}_{Y} X + [X,Y].
\end{equation*}

By virtue of connection $\overset{1}{\nabla}$ and its dual connection $\overset{2}{\nabla}$, one can define symmetric connection $\overset{0}{\nabla}$ given with
\begin{equation}\label{eq:symmetricconnectionnabla0}
	\overset{0}{\nabla}_{X} Y = \frac{1}{2} (\overset{1}{\nabla}_{X} Y + \overset{2}{\nabla}_{X} Y).
\end{equation}

Based on symmetric connection $\overset{0}{\nabla}$ and non-symmetric connections $\overset{1}{\nabla}$ and $\overset{2}{\nabla}$, we can observe six linearly independent curvature tensors (see \cite{mincic1979})
\begin{align*}
	\overset{\theta}{R} (X,Y)Z & = \overset{\theta}{\nabla}_{X}\overset{\theta}{\nabla}_{Y} Z - \overset{\theta}{\nabla}_{Y}\overset{\theta}{\nabla}_{X} Z - \overset{\theta}{\nabla}_{[X,Y]} Z, \; \theta=0,1,2, \\
	\overset{3}{R} (X,Y)Z & = \overset{2}{\nabla}_{X}\overset{1}{\nabla}_{Y} Z - \overset{1}{\nabla}_{Y}\overset{2}{\nabla}_{X} Z + \overset{2}{\nabla}_{\overset{1}{\nabla}_{Y} X} Z - \overset{1}{\nabla}_{\overset{2}{\nabla}_{X} Y} Z, \\
	\overset{4}{R} (X,Y)Z & = \overset{2}{\nabla}_{X}\overset{1}{\nabla}_{Y} Z - \overset{1}{\nabla}_{Y}\overset{2}{\nabla}_{X} Z + \overset{2}{\nabla}_{\overset{2}{\nabla}_{Y} X} Z - \overset{1}{\nabla}_{\overset{1}{\nabla}_{X} Y} Z, \\
	\overset{5}{R} (X,Y)Z & = \frac{1}{2} ( \overset{1}{\nabla}_{X}\overset{1}{\nabla}_{Y} Z - \overset{2}{\nabla}_{Y}\overset{1}{\nabla}_{X} Z + \overset{2}{\nabla}_{X}\overset{2}{\nabla}_{Y} Z - \overset{1}{\nabla}_{Y}\overset{2}{\nabla}_{X} Z  - \overset{1}{\nabla}_{[X,Y]} Z - \overset{2}{\nabla}_{[X,Y]} Z ).
\end{align*}

Riemannian curvature tensor $\overset{g}{R}$ with respect to Levi-Civita connection $\overset{g}{\nabla}$ is defined with equation
\begin{equation*}
	\overset{g}{R} (X,Y)Z  = \overset{g}{\nabla}_{X}\overset{g}{\nabla}_{Y} Z - \overset{g}{\nabla}_{Y}\overset{g}{\nabla}_{X} Z - \overset{g}{\nabla}_{[X,Y]} Z.
\end{equation*}

Linear connection $\overset{1}{\nabla}$ on a generalized Riemannian manifold is determined with torsion tensor and covariant derivative of symmetric part $g$.

\begin{theorem}\label{thm:connectioninGRn}\cite{ivanov2016}
	Let $(\mathcal{M}, G=g+F)$ be a generalized Riemannian manifold and $\overset{g}{\nabla}$ be a Levi-Civita connection of $g$. Let $\overset{1}{\nabla}$ be a linear connection with torsion $\overset{1}{T}$ and denote the covariant derivative of the symmetric part $g$ of tensor $G$ with $\overset{1}{\nabla} g$. Then the connection $\overset{1}{\nabla}$ is uniquely determined by the following formula
	\begin{equation}\label{eq:linearconnectionnabla1}
		\begin{split}
			g(\overset{1}{\nabla}_{X} Y, Z) = & g(\overset{g}{\nabla}_{X} Y, Z) + \frac{1}{2} ( \overset{1}{T}(X,Y,Z) + \overset{1}{T}(Z,X,Y)- \overset{1}{T}(Y,Z,X) ) \\
			& -\frac{1}{2} ( (\overset{1}{\nabla}_{X}g)(Y,Z) + (\overset{1}{\nabla}_{Y}g)(Z,X) - (\overset{1}{\nabla}_{Z}g)(Y,X) ).
		\end{split}
	\end{equation}
	The covariant derivative $\overset{1}{\nabla}F$ of the skew symmetric part $F$ of tensor $G$ is given by
	\begin{equation}\label{eq:covderofF}
		\begin{split}
			(\overset{1}{\nabla}_{X}F)(Y,Z) = & (\overset{g}{\nabla}_{X}F)(Y,Z)  + \frac{1}{2} ( \overset{1}{T}(X,Y,AZ) + \overset{1}{T}(Z,X,AY)) \\
			& +\frac{1}{2} (\overset{1}{T}(AZ,X,Y) + \overset{1}{T}(AZ,Y,X) + \overset{1}{T}(X,AY,Z) + \overset{1}{T}(Z,AY,X)  ) \\
			& + \frac{1}{2} ( (\overset{1}{\nabla}_{X}g)(AY,Z) - (\overset{1}{\nabla}_{X}g)(Y,AZ) - (\overset{1}{\nabla}_{Y}g)(AZ,X) ) \\
			& + \frac{1}{2} ( (\overset{1}{\nabla}_{Z}g)(AY,X) + (\overset{1}{\nabla}_{AZ}g)(Y,X) - (\overset{1}{\nabla}_{AY}g)(Z,X) ).
		\end{split}
	\end{equation}
	In particular, the exterior derivative $\mathrm{d}F$ of the skew symmetric part $F$ satisfies
	\begin{equation}\label{eq:exteriorderivativeofF1}
		\begin{split}
			\mathrm{d}F(X,Y,Z) = & -\overset{1}{T}(X,Y,AZ) - \overset{1}{T}(Y,Z,AX)- \overset{1}{T}(Z,X,AY) \\ 
			& + (\overset{1}{\nabla}_{X}F)(Y,Z) + (\overset{1}{\nabla}_{Y}F)(Z,X) + (\overset{1}{\nabla}_{Z}F)(X,Y).
		\end{split}
	\end{equation}
	Conversely, any three tensors $\overset{1}{T}$, $\overset{1}{\nabla} g$, $\overset{1}{\nabla} F$ satisfying (\ref{eq:covderofF}) determine a unique linear connection via (\ref{eq:linearconnectionnabla1}).
\end{theorem}

\section{Quarter-symmetric non-metric connection}

In paper \cite{ivanov2016}, the authors proved that connection $\overset{1}{\nabla}$ preserves the basic $G$ if and only if it preserves both its symmetric part $g$ and its skew-symmetric part $F$, i.e. 
\begin{equation*}
	\overset{1}{\nabla}G=0 \Leftrightarrow \overset{1}{\nabla}g=\overset{1}{\nabla}F=0 \Leftrightarrow \overset{1}{\nabla}g=\overset{1}{\nabla}A=0.
\end{equation*}
If connection $\overset{1}{\nabla}$ preserves symmetric metric $g$, i.e. $\overset{1}{\nabla}g=0$, then it is called the \textit{metric} connection; otherwise, it is \textit{non-metric}. According to the previous equation, if linear connection $\overset{1}{\nabla}$ preserves the basic tensor $G$ then it is called the \textit{metric $A$-connection}. This paper aims to research the quarter-symmetric non-metric connection on a generalized Riemannian manifold. Recently, many papers studied quarter-symmetric connections on various manifolds (for example, see \cite{dizhao2022,pal2019, chaubeyde2019,dezhao2020,dede2011,khan2020,khan2023,han2013,tang2018,tang2019,hui2018}). Paper \cite{tang2023} deals with the application of quarter-symmetric connection to probability theory, i.e. to the financial market.

In the following theorem, we will prove the existence of the connection (\ref{eq:Q-Sconnection1}) with coefficients $a=-b=\frac{1}{2}$, i.e. we will show that equation (\ref{eq:linearconnectionnabla1}) is satisfied. Analogously, the existence of connection (\ref{eq:Q-Sconnection1}) with arbitrary coefficients $a,b \in \mathbb{R}$, $a\ne b$, can be proven, but for simpler calculation purposes, we chose $a=-b=\frac{1}{2}$. 
\begin{theorem}
	Let $(\mathcal{M}, G=g+F)$ be a generalized Riemannian manifold and $\overset{g}{\nabla}$ be a Levi-Civita connection of symmetric part $g$. Then there exists a unique linear connection $\overset{1}{\nabla}$ given with equation
	\begin{equation}\label{eq:Q-Snmetric}
		\overset{1}{\nabla}_{X} Y = \overset{g}{\nabla}_{X} Y + \frac{1}{2}\pi(Y) A X - \frac{1}{2}\pi(X) A Y, 
	\end{equation}
	whose torsion tensor is given with
	\begin{equation}\label{eq:ttofQ-Snmetric}
		\overset{1}{T}(X,Y)= \pi(Y) A X - \pi(X) A Y,
	\end{equation}
	and which satisfies
	\begin{equation}\label{eq:covderofgQ-Snmetric}
		(\overset{1}{\nabla}_{X} g)(Y,Z) = -\frac{1}{2} (\pi(Y)F(X,Z) + \pi(Z)F(X,Y)),
	\end{equation}
	where $\pi$ is a 1-form associated with vector field $P$, i.e. $\pi(X)=g(X,P)$ and $A$ is a (1,1) type tensor field associated with skew-symmetric part $F$ of basic tensor $G$, i.e. $F (X,Y)= g(A X, Y)$.
\end{theorem}
\begin{proof} Linear connection $\overset{1}{\nabla}$ can be written in the form
	\begin{equation*}
		\overset{1}{\nabla}_{X} Y = \overset{g}{\nabla}_{X} Y + H(X,Y),
	\end{equation*}
	i.e. 
	\begin{equation*}
		g(\overset{1}{\nabla}_{X} Y, Z) =  g(\overset{g}{\nabla}_{X} Y, Z) + H(X,Y,Z),
	\end{equation*}
	and based on equation (\ref{eq:linearconnectionnabla1}) we have
	\begin{equation*}
		\begin{split}
			H(X,Y,Z):=g(H(X,Y),Z) = & \frac{1}{2} ( \overset{1}{T}(X,Y,Z) + \overset{1}{T}(Z,X,Y)- \overset{1}{T}(Y,Z,X) ) \\
			& -\frac{1}{2} ( (\overset{1}{\nabla}_{X}g)(Y,Z) + (\overset{1}{\nabla}_{Y}g)(Z,X) - (\overset{1}{\nabla}_{Z}g)(Y,X) ).
		\end{split}
	\end{equation*}
	Now we will determine tensor $H$ so that the previous equation is satisfied. Equation (\ref{eq:covderofgQ-Snmetric}) implies the following relations
	\begin{align*}
		(\overset{1}{\nabla}_{Y} g)(Z,X) = & - \frac{1}{2} (\pi(Z)F(Y,X) +\pi(X)F(Y,Z)), 
		\\
		(\overset{1}{\nabla}_{Z} g)(Y,X) = & - \frac{1}{2} (\pi(Y)F(Z,X) +\pi(X)F(Z,Y)).
	\end{align*} 
	From (\ref{eq:ttofQ-Snmetric}), it follows
	\begin{equation*}
		\overset{1}{T}(X,Y,Z)= \pi(Y) F(X,Z) - \pi(X) F(Y,Z),
	\end{equation*}
	and further
	\begin{align*} 
		\overset{1}{T}(Z,X,Y) = & \pi(X) F(Z,Y) - \pi(Z) F(X,Y), \\
		\overset{1}{T}(Y,Z,X) = & \pi(Z) F(Y,X) - \pi(Y) F(Z,X).
	\end{align*}
	Combining the previous six equations and equation (\ref{eq:covderofgQ-Snmetric}) we have
	\begin{equation*}
		H(X,Y,Z) = \frac{1}{2}\pi(Y)F(X,Z)- \frac{1}{2}\pi(X)F(Y,Z),
	\end{equation*}
	from which
	\begin{equation*}
		H(X,Y) = \frac{1}{2}\pi(Y)AX- \frac{1}{2}\pi(X)AY, 
	\end{equation*}
	and thereby, we proved the theorem.
\end{proof}

According to equations (\ref{eq:ttofQ-Snmetric}) and (\ref{eq:covderofgQ-Snmetric}), connection $\overset{1}{\nabla}$ given with (\ref{eq:Q-Snmetric}) is called a \textit{quarter-symmetric non-metric connection}. The 1-form $\pi$ is a \textit{generator} of that connection.

For covariant derivative of skew-symmetric part $F$ we have
\begin{equation*}
	\begin{split}
		(\overset{1}{\nabla}_{X} F)(Y,Z) = & XF(Y,Z)- F(\overset{1}{\nabla}_{X} Y, Z) - F( Y,\overset{1}{\nabla}_{X} Z) \\
		= & (\overset{g}{\nabla}_{X} F)(Y,Z) - \frac{1}{2}F(\pi(Y)AX- \pi(X)AY) - \frac{1}{2}F(Y,\pi(Z)AX- \pi(X)AZ), 
	\end{split}
\end{equation*}
i.e.
\begin{equation}\label{eq:covderofF2}
	(\overset{1}{\nabla}_{X} F)(Y,Z) = (\overset{g}{\nabla}_{X} F)(Y,Z) + \frac{1}{2} (\pi(Y)g(AX,AZ)  - \pi(Z)g(AX,AY) ),
\end{equation}
from which we see that connection (\ref{eq:Q-Snmetric}) does not preserve the tensor $F$. By adding equations (\ref{eq:covderofgQ-Snmetric}) and (\ref{eq:covderofF2}), we obtain the covariant derivative of basic tensor $G$
\begin{equation*}
	\begin{split}
		(\overset{1}{\nabla}_{X} G)(Y,Z) & = (\overset{1}{\nabla}_{X} g)(Y,Z) + (\overset{1}{\nabla}_{X} F)(Y,Z) \\
		& = (\overset{g}{\nabla}_{X} F)(Y,Z) + \frac{1}{2} (\pi(Y)(g(AX,AZ) - F(X,Z))  - \pi(Z)(g(AX,AY)+F(X,Y)) ).
	\end{split}
\end{equation*}

The covariant derivative of tensor $A$ is given by 
\begin{align*}
	(\overset{1}{\nabla}_{X} A )Y & = \overset{1}{\nabla}_{X} A Y - A(\overset{1}{\nabla}_{X} Y) \\
	& = \overset{g}{\nabla}_{X} A Y + \frac{1}{2}\pi(AY)AX - \frac{1}{2}\pi(X)A^2Y - A(\overset{g}{\nabla}_{X} Y \frac{1}{2}\pi(Y)AX - \frac{1}{2}\pi(X)AY),
\end{align*}
i.e.
\begin{equation*}
	(\overset{1}{\nabla}_{X} A )Y = (\overset{g}{\nabla}_{X} A )Y + \frac{1}{2}(\pi(AY)AX-\pi(Y)A^2X).
\end{equation*}

For covariant derivative of 1-form $\pi$ with respect to quarter-symmetric non-metric connection (\ref{eq:Q-Snmetric}), the following equation holds 
\begin{equation}\label{eq:covderivofpi-QSnMC}
	(\overset{1}{\nabla}_{X} \pi )(Y) = (\overset{g}{\nabla}_{X} \pi )(Y) + \frac{1}{2} \pi(X)\pi(AY) - \frac{1}{2} \pi(AX)\pi(Y).
\end{equation}

A vector field $P$ is a \textit{Killing} if ${\mathcal{L}}_{P} g=0$, where $\mathcal{L}$ is the Lie derivative with respect to the Levi-Civita connection. Based on the previous equation, we will prove the following theorems.

\begin{theorem}
	Let $(\mathcal{M}, G=g+F)$ be a generalized Riemannian manifold with the quarter-symmetric non-metric connection $\overset{1}{\nabla}$ given with (\ref{eq:Q-Snmetric}). A necessary and sufficient condition for the vector $P$ to be Killing is that 
	\begin{equation*}
		(\overset{1}{\nabla}_{X} \pi )(Y) + (\overset{1}{\nabla}_{Y} \pi )(X)=0
	\end{equation*}
	holds.
\end{theorem}
\begin{proof}
	The Lie derivative of $g$ with respect to the Levi-Civita connection is given by
	\begin{equation*}
		({\mathcal{L}}_{P} g) (X,Y) = g(\overset{g}{\nabla}_{X} P,Y) + g(X, \overset{g}{\nabla}_{Y} P)
	\end{equation*}
	i.e.
	\begin{equation*}
		({\mathcal{L}}_{P} g) (X,Y) = (\overset{g}{\nabla}_{X} \pi )(Y) + (\overset{g}{\nabla}_{Y} \pi )(X).
	\end{equation*}
	Based on equation (\ref{eq:covderivofpi-QSnMC}), we obtain
	\begin{equation*}
		({\mathcal{L}}_{P} g) (X,Y)= (\overset{1}{\nabla}_{X} \pi )(Y) + (\overset{1}{\nabla}_{Y} \pi )(X).
	\end{equation*}
\end{proof}

\begin{theorem}
	Let $(\mathcal{M}, G=g+F)$ be a generalized Riemannian manifold with the quarter-symmetric non-metric connection $\overset{1}{\nabla}$ given with (\ref{eq:Q-Snmetric}). Then the following relation holds
	\begin{equation*}
		(\overset{1}{\mathcal{L}}_{P} g) (X,Y) = ({\mathcal{L}}_{P} g) (X,Y) + \pi(X)\pi(AY) + \pi(AX)\pi(Y),
	\end{equation*}
	where $\overset{1}{\mathcal{L}}$ and ${\mathcal{L}}$ denote the Lie derivatives with respect to $\overset{1}{\nabla}$ and $\overset{g}{\nabla}$, respectively.
\end{theorem}
\begin{proof}
	The Lie derivative of $g$ with respect to the connection $\overset{1}{\nabla}$ is given by
	\begin{align*}
		(\overset{1}{\mathcal{L}}_{P} g) (X,Y) & = Pg(X,Y) - g(\overset{1}{\nabla}_{P}X - \overset{1}{\nabla}_{X} P,Y) - g(X, \overset{1}{\nabla}_{P}Y -  \overset{1}{\nabla}_{Y} P) \\
		& = (\overset{1}{\nabla}_{P} g)(X,Y) + g(\overset{1}{\nabla}_{X} P,Y) + g(X, \overset{1}{\nabla}_{Y} P).
	\end{align*}
	By using equation (\ref{eq:covderofgQ-Snmetric}), we have
	\begin{equation*}
		(\overset{1}{\nabla}_{P} g)(X,Y) = \frac{1}{2} (\pi(X)\pi(AY) + \pi(AX)\pi(Y)).
	\end{equation*}
	If we consider the previous two equations and (\ref{eq:Q-Snmetric}), we get
	\begin{align*}
		(\overset{1}{\mathcal{L}}_{P} g) (X,Y) & = g(\overset{g}{\nabla}_{X} P,Y) + g(X, \overset{g}{\nabla}_{Y} P) +  \pi(X)\pi(AY) + \pi(AX)\pi(Y)
		\\ & = ({\mathcal{L}}_{P} g) (X,Y) + \pi(X)\pi(AY) + \pi(AX)\pi(Y).
	\end{align*}
\end{proof}

Directly based on the previous theorem, we have the following statement.
\begin{corollary}
	Let $(\mathcal{M}, G=g+F)$ be a generalized Riemannian manifold with the quarter-symmetric non-metric connection $\overset{1}{\nabla}$ given with (\ref{eq:Q-Snmetric}). If the vector field $P$ is Killing then 
	\begin{equation*}
		(\overset{1}{\mathcal{L}}_{P} g) (X,Y) = \pi(X)\pi(AY) + \pi(AX)\pi(Y).
	\end{equation*}
\end{corollary}

The dual connection $\overset{2}{\nabla}$ of quarter-symmetric non-metric connection $\overset{1}{\nabla}$ is determined with equation
\begin{equation}\label{eq:Q-S-nmetric2}
	\overset{2}{\nabla}_{X} Y = \overset{g}{\nabla}_{X} Y  - \frac{1}{2} \pi(Y) A X + \frac{1}{2} \pi(X)AY.
\end{equation}

Connection $\overset{2}{\nabla}$ is also non-metric and satisfies the following relations
\begin{align*}
	(\overset{2}{\nabla}_{X} g)(Y,Z) & = \frac{1}{2} (\pi(Y)F(X,Z) + \pi(Z)F(X,Y)), \\
	(\overset{2}{\nabla}_{X} F)(Y,Z) & = (\overset{g}{\nabla}_{X} F)(Y,Z) - \frac{1}{2} (\pi(Y)g(AX,AZ)  - \pi(Z)g(AX,AY) ), \\
	(\overset{2}{\nabla}_{X} G)(Y,Z) & =  (\overset{g}{\nabla}_{X} F)(Y,Z) - \frac{1}{2} (\pi(Y)(g(AX,AZ) - F(X,Z))  - \pi(Z)(g(AX,AY)+F(X,Y)) ), \\
	(\overset{2}{\nabla}_{X} A )Y & = (\overset{g}{\nabla}_{X} A )Y - \frac{1}{2}(\pi(AY)AX-\pi(Y)A^2X).
\end{align*}

Based on equations (\ref{eq:symmetricconnectionnabla0}) and (\ref{eq:Q-S-nmetric2}), we conclude that the symmetric connection $\overset{0}{\nabla}$ coincides with Levi-Civita connection $\overset{g}{\nabla}$, i.e. the following holds 
\begin{equation}\label{eq:Q-S-nmetric0}
	\overset{0}{\nabla}_{X} Y = \overset{g}{\nabla}_{X} Y.
\end{equation}

In the following assertion, we state some identities for torsion tensor $\overset{1}{T}$ given with (\ref{eq:ttofQ-Snmetric}).

\begin{theorem}
	The torsion tensor of quarter-symmetric non-metric connection (\ref{eq:Q-Snmetric}) in the generalized Riemannian manifold satisfies the following relations
	\begin{align} \label{eq:identity1}
		\underset{XYZ}{\sigma} \overset{1}{T}(X,Y,Z) & = -2 \underset{XYZ}{\sigma} \pi(X)F(Y,Z),
		\\ \label{eq:identity2}
		2 \underset{XYZ}{\sigma} \pi(X)F(AY,AZ) & = - \underset{XYZ}{\sigma} (\overset{1}{T}(AX,Y,AZ)  + \overset{1}{T}(X,AY,AZ)),
		\\
		\label{eq:sigmaT1(T1)}
		\underset{XYZ}{\sigma}  \overset{1}{T}( \overset{1}{T}(X,Y),Z) &  = \underset{XYZ}{\sigma}    \pi (X) (\pi(AY) A Z - \pi(AZ) A Y ) = \underset{XYZ}{\sigma}    \pi (AX)\overset{1}{T}(Y,Z),	
		\\ \label{eq:identity3}
		\underset{XYZ}{\sigma} \overset{1}{T}(X,Y,AZ) & = 0, 	
		\\ 
		\underset{XYZ}{\sigma} \overset{1}{T}(AX,AY,Z) & = 0, 
	\end{align}
	where $\underset{XYZ}{\sigma}$ denote the cyclic sum with respect to the vector fields $X,Y,Z$.	
\end{theorem}
\begin{proof} Using torsion tensor (\ref{eq:ttofQ-Snmetric}), the skew-symmetry of tensor $F$ and equation (\ref{eq:FgA}), the previous identities are easily proved.
\end{proof}

\begin{remark}
	The identities from the previous theorem are true for any quarter-symmetric connection with a torsion tensor (\ref{eq:ttofQ-Snmetric}).
\end{remark}

Let $\overset{1}{\mathrm{d}}{F}$ be the exterior derivative of $F$ with respect to quarter-symmetric non-metric connection $\overset{1}{\nabla}$, i.e.
\begin{equation*}
	\overset{1}{\mathrm{d}}{F} (X,Y,Z) = (\overset{1}{\nabla}_{X}F)(Y,Z) + (\overset{1}{\nabla}_{Y}F)(Z,X) + (\overset{1}{\nabla}_{Z}F)(X,Y).
\end{equation*}
By the cyclic sum of equation (\ref{eq:covderofF2}) we obtain
\begin{equation}\label{eq:dF1=dF}
	\begin{split}
		\overset{1}{\mathrm{d}}{F}(X,Y,Z) & = (\overset{1}{\nabla}_{X}F)(Y,Z) + (\overset{1}{\nabla}_{Y}F)(Z,X) + (\overset{1}{\nabla}_{Z}F)(X,Y) \\
		& = (\overset{g}{\nabla}_{X}F)(Y,Z) + (\overset{g}{\nabla}_{Y}F)(Z,X) + (\overset{g}{\nabla}_{Z}F)(X,Y) = \mathrm{d}F(X,Y,Z).
	\end{split}
\end{equation}
\begin{theorem}\label{thm:dF1=dF}
	Let $(\mathcal{M}, G=g+F)$ be a generalized Riemannian manifold with the quarter-symmetric non-metric connection $\overset{1}{\nabla}$ given with (\ref{eq:Q-Snmetric}). The exterior derivative of skew-symmetric part $F$ with respect to Levi-Civita connection $\overset{g}{\nabla}$ coincides with that of skew-symmetric part $F$ with respect to quarter-symmetric non-metric connection $\overset{1}{\nabla}$.
\end{theorem}

In other words, the 2-form $F$ is closed with respect to the Levi-Civita connection if and only if it is closed with respect to the quarter-symmetric non-metric connection (\ref{eq:Q-Snmetric}). Based on the fact that the even-dimensional manifold with closed form $F$ (with respect to the Levi-Civita connection) is a symplectic manifold, we have the following statement.

\begin{theorem}
	Let $(\mathcal{M}, G=g+F)$ be a generalized Riemannian manifold with the quarter-symmetric non-metric connection (\ref{eq:Q-Snmetric}) and $F$ be a non-degenerate 2-form. If the manifold $\mathcal{M}$ is even-dimensional, then the pair $(\mathcal{M}, F)$ is a symplectic manifold if and only if the skew-symmetric part $F$ is closed with respect to the quarter-symmetric non-metric connection (\ref{eq:Q-Snmetric}).
\end{theorem}

The Nijenhuis tensor (or the torsion) of tensor $A$ is the type of (1,2) tensor and plays an important role in almost complex and almost para-complex geometry. If tensor $A$ has the property $A^2=-1$ (respectively, $A^2=1$), then it is called an \textit{almost complex structure} (respectively, \textit{almost para-complex structure}). The vanishing of the Nijenhuis tensor of the almost complex structure is equivalent with the integrability of the almost complex structure \cite{newlander1957}. 

In the generalized Riemannian manifold, the Nijenhus tensor can be represent as follows (see \cite{ivanov2016}, equation (2.17))
\begin{equation}\label{eq:N(X,Y)}
	N(X,Y)= \overset{1}{N}(X,Y) -\overset{1}{T}(AX,AY) - A^2 \overset{1}{T}(X,Y) + A\overset{1}{T}(AX,Y) + A\overset{1}{T}(X,AY),
\end{equation}
where $\overset{1}{N}$ denotes the (1,2) type tensor given with equation
\begin{equation}\label{eq:N1(X,Y)}
	\overset{1}{N}(X,Y) = (\overset{1}{\nabla}_{AX} A)Y - (\overset{1}{\nabla}_{AY} A)X - A(\overset{1}{\nabla}_{X} A)Y + A(\overset{1}{\nabla}_{Y} A)X.
\end{equation}

Since for torsion tensor (\ref{eq:ttofQ-Snmetric}) it holds that
\begin{equation*}
	-\overset{1}{T}(AX,AY) - A^2 \overset{1}{T}(X,Y) + A\overset{1}{T}(AX,Y) + A\overset{1}{T}(X,AY)=0,
\end{equation*}
from equation (\ref{eq:N(X,Y)}), it follows
\begin{equation*}
	N(X,Y)= \overset{1}{N}(X,Y).
\end{equation*}
With this equation, we proved the following theorem.
\begin{theorem}\label{thm:N=N1}
	In the generalized Riemannian manifold with quarter-symmetric non-metric connection (\ref{eq:Q-Snmetric}), Nijenhuis tensor $N$ coincides with  tensor $\overset{1}{N}$ given with (\ref{eq:N1(X,Y)}).
\end{theorem}

\begin{remark} Equation (\ref{eq:dF1=dF}) can be proved using equations (\ref{eq:exteriorderivativeofF1}) and (\ref{eq:identity3}).
\end{remark}

\section{Properties of curvature tensors}

In this section, we will deal with linearly independent curvature tensors and their properties. The curvature tensor $\overset{1}{R}$ of quarter-symmetric non-metric connection $\overset{1}{\nabla}$, given with (\ref{eq:Q-Snmetric}), can be expressed by the following relation
\begin{equation}\label{eq:R1XYZQ-Snm}
	\begin{split}
		\overset{1}{R} (X,Y)Z = \overset{g}{R} (X,Y)Z  & +  \frac{1}{2}\overset{1}{\alpha}(X,Z)AY - \frac{1}{2}\overset{1}{\alpha}(Y,Z)AX - \frac{1}{2}\overset{1}{\beta}(X,Y)AZ \\
		& -  \frac{1}{2}\overset{1}{\gamma}(X,Z)\pi(Y) + \frac{1}{2}\overset{1}{\gamma}(Y,Z)\pi(X) + \frac{1}{2}\overset{1}{\delta}(X,Y)\pi(Z),
	\end{split}
\end{equation}
where $\overset{1}{\alpha}$, $\overset{1}{\beta}$ are (0,2) type tensor fields and $\overset{1}{\gamma}$, $\overset{1}{\delta}$ are (1,2) type tensor fields given with equation
\begin{align}
	\label{eq:alpha1}
	\overset{1}{\alpha}(X,Y) & = (\overset{g}{\nabla}_{X} \pi )(Y) + \frac{1}{2} \pi(X)\pi(AY) - \frac{1}{2} \pi(AX)\pi(Y) = (\overset{1}{\nabla}_{X} \pi )(Y), \\
	\label{eq:beta1}
	\overset{1}{\beta}(X,Y) & = (\overset{g}{\nabla}_{X} \pi )(Y) - (\overset{g}{\nabla}_{Y} \pi )(X), \\
	\label{eq:gamma1}
	\overset{1}{\gamma}(X,Y) & = (\overset{g}{\nabla}_{X} A )Y -\frac{1}{2} \pi(Y)A^2X, \\
	\label{eq:delta1}
	\overset{1}{\delta}(X,Y) & = (\overset{g}{\nabla}_{X} A )Y - (\overset{g}{\nabla}_{Y} A )X.
\end{align}

According to equation (\ref{eq:Q-S-nmetric0}), the curvature tensor $\overset{0}{R}$ of symmetric connection $\overset{0}{\nabla}$ coincides with Riemannian curvature tensor $\overset{g}{R}$ of the Levi-Civita connection, i.e. $\overset{0}{R}= \overset{g}{R}$. Relations for other linearly independent curvature tensors are presented in the following theorem. 

\begin{theorem}\label{thm:curvaturetensorsofQSnMC2}
	Let $(\mathcal{M}, G=g+F)$ be a generalized Riemannian manifold with the quarter-symmetric non-metric connection (\ref{eq:Q-Snmetric}). Curvature tensors $\overset{\theta}{R} (X,Y)Z$, $\theta = 2, ..., 5$ and Riemannian curvature tensor $\overset{g}{R} (X,Y)Z$ satisfy the following relations
	\begin{align}
		\label{eq:R2XYZQ-Snm1}
		\begin{split}
			\overset{2}{R} (X,Y)Z = & \overset{g}{R} (X,Y)Z  -  \frac{1}{2}\overset{2}{\alpha}(X,Z)AY + \frac{1}{2}\overset{2}{\alpha}(Y,Z)AX + \frac{1}{2}\overset{1}{\beta}(X,Y)AZ \\
			& +  \frac{1}{2}\overset{2}{\gamma}(X,Z)\pi(Y) - \frac{1}{2}\overset{2}{\gamma}(Y,Z)\pi(X) - \frac{1}{2}\overset{1}{\delta}(X,Y)\pi(Z),
		\end{split} \\
		\label{eq:R3XYZQ-Snm1}
		\begin{split}
			\overset{3}{R} (X,Y)Z = & \overset{g}{R} (X,Y)Z + \frac{1}{2}\overset{2}{\alpha}(X,Z)AY + \frac{1}{2}\overset{1}{\alpha}(Y,Z)AX - \frac{1}{2}(\overset{2}{\alpha}(X,Y) + \overset{1}{\alpha}(Y,X))AZ \\
			& -  \frac{1}{2}\overset{1}{\gamma}(X,Z)\pi(Y) - \frac{1}{2}\overset{1}{\gamma}(Y,Z)\pi(X) + \frac{1}{2}(\overset{1}{\delta}(X,Y) + 2\overset{1}{\gamma} (Y,X))\pi(Z),
		\end{split} 
		\\
		\label{eq:R4XYZQ-Snm1}
		\begin{split}
			\overset{4}{R} (X,Y)Z = & \overset{g}{R} (X,Y)Z  + \frac{1}{2}\overset{2}{\alpha}(X,Z)AY + \frac{1}{2}\overset{1}{\alpha}(Y,Z)AX - \frac{1}{2}(\overset{1}{\alpha}(X,Y) + \overset{2}{\alpha}(Y,X))AZ \\
			& -  \frac{1}{2}\overset{1}{\gamma}(X,Z)\pi(Y) - \frac{1}{2}\overset{1}{\gamma}(Y,Z)\pi(X) + \frac{1}{2}(\overset{1}{\delta}(X,Y) + 2\overset{1}{\gamma} (Y,X))\pi(Z) \\
			& - \pi (Z) (\pi (Y) A^2 X - \pi (X) A^2 Y  ),
		\end{split} 		
	\end{align}
	\begin{align}
		\label{eq:R5XYZQ-Snm1}
		\begin{split}
			\overset{5}{R} (X,Y)Z = & \overset{g}{R} (X,Y)Z  + \frac{1}{4} \pi(X) (\pi(Y) A^2 Z - \pi(AZ) AY) + \frac{1}{4} \pi(Y) (\pi (X) A^2 Z - \pi(AZ) AX) \\
			& - \frac{1}{4} \pi(Z) (\pi(X) A^2Y + \pi(Y) A^2X - \pi(AX)AY- \pi(AY)AX), 
		\end{split}
	\end{align}
	where $\overset{1}{\alpha}$, $\overset{1}{\beta}$, $\overset{1}{\gamma}$, $\overset{1}{\delta}$ given with (\ref{eq:alpha1}), (\ref{eq:beta1}), (\ref{eq:gamma1}), (\ref{eq:delta1}), respectively, and $\overset{2}{\alpha}$, $\overset{2}{\gamma}$ given with
	\begin{align}
		\label{eq:alpha2}
		\overset{2}{\alpha}(X,Y) & = (\overset{g}{\nabla}_{X} \pi )(Y) - \frac{1}{2} \pi(X)\pi(AY) + \frac{1}{2} \pi(AX)\pi(Y), \\
		\label{eq:gamma2}
		\overset{2}{\gamma}(X,Y) & = (\overset{g}{\nabla}_{X} A )Y +\frac{1}{2} \pi(Y)A^2X.
	\end{align}
\end{theorem}
\begin{proof}
	We will prove the relation for curvature tensor $\overset{3}{R}$. Relation between curvature tensors $\overset{3}{R}$ and $\overset{1}{R}$ is given with formula  (see \cite{zlatanovic2021})
	\begin{equation*}
		\overset{3}{R} (X,Y)Z = \overset{1}{R} (X,Y)Z + (\overset{1}{\nabla}_{Y} \overset{1}{T})(X,Z).
	\end{equation*}
	For the covariant derivative of torsion tensor $\overset{1}{T}$ the following equation holds    
	\begin{equation*}
		\begin{split}
			(\overset{1}{\nabla}_{X} 	\overset{1}{T})(Y,Z) & =  AY (  \overset{1}{\nabla}_{X} \pi  )(Z) + \pi(Z) (  \overset{1}{\nabla}_{X} A)Y - AZ (\overset{1}{\nabla}_{X} \pi  )(Y) - \pi(Y) (  \overset{1}{\nabla}_{X}A)Z \\
			& = \overset{1}{\alpha}(X,Z)AY -\overset{1}{\alpha}(X,Y)AZ + \pi(Z)\overset{1}{\gamma}(X,Y) - \pi(Y)\overset{1}{\gamma}(X,Z) \\
			& \quad + \frac{1}{2} (\pi(AY)\pi(Z)-\pi(Y)\pi(AZ))AX,
		\end{split}
	\end{equation*}
	where $\overset{1}{\alpha}$ and $\overset{1}{\gamma}$ are given with equations (\ref{eq:alpha1}) and (\ref{eq:gamma1}). By combining the previous two equations and (\ref{eq:R1XYZQ-Snm}), we get (\ref{eq:R3XYZQ-Snm1}).
\end{proof}

Below we will study the transformations of Levi-Civita connection $\overset{g}{\nabla}$ on connections $\overset{1}{\nabla}$ and $\overset{2}{\nabla}$. We will determine the conditions for the Riemannian curvature tensor to be invariant under these transformations.

If we define (1,3) tensor $\overset{1}{M} (X,Y)Z$ as follows
\begin{equation}\label{eq:M1XYZQ-Snm}
	\begin{split}
		\overset{1}{M} (X,Y)Z =  \overset{1}{\alpha}(X,Z)AY  -  \overset{1}{\gamma}(X,Z)\pi(Y)  - (\overset{g}{\nabla}_{X} \pi )(Y) AZ + \pi(Z) (\overset{g}{\nabla}_{X} A )Y,
	\end{split}
\end{equation}
where $\overset{1}{\alpha}$ and $\overset{1}{\gamma}$ are given with equations (\ref{eq:alpha1}) and (\ref{eq:gamma1}), then the curvature tensor $\overset{1}{R}$ takes the form
\begin{equation*}
	\overset{1}{R} (X,Y)Z = \overset{g}{R} (X,Y)Z + \frac{1}{2}\overset{1}{M} (X,Y)Z - \frac{1}{2}\overset{1}{M} (Y,X)Z,
\end{equation*}
from which we see that curvature tensor $\overset{1}{R}$ is equal to Riemannian curvature tensor $\overset{g}{R}$ if and only if tensor $\overset{1}{M}$ is symmetric with respect to vectors $X$ and $Y$.
\begin{theorem}
	The Riemannian curvature tensor $\overset{g}{R} (X,Y)Z$ is invariant under connection transformation $\overset{g}{\nabla} \rightarrow \overset{1}{\nabla}$ if and only if tensor $\overset{1}{M} (X,Y)Z$ is symmetric with respect to $X$ and $Y$, where $\overset{1}{M}$ is given with (\ref{eq:M1XYZQ-Snm}), $\overset{g}{\nabla}$ denotes Levi-Civita connection and $\overset{1}{\nabla}$ denotes the quarter-symmetric non-metric connection (\ref{eq:Q-Snmetric}).
\end{theorem}

Similarly, curvature tensor $\overset{2}{R}$ can be written in the form
\begin{equation*}
	\overset{2}{R} (X,Y)Z = \overset{g}{R} (X,Y)Z - \frac{1}{2}\overset{2}{M} (X,Y)Z + \frac{1}{2}\overset{2}{M} (Y,X)Z,
\end{equation*}
where
\begin{equation}\label{eq:M2XYZQ-Snm}
	\begin{split}
		\overset{2}{M} (X,Y)Z =  \overset{2}{\alpha}(X,Z)AY  -  \overset{2}{\gamma}(X,Z)\pi(Y)  - (\overset{g}{\nabla}_{X} \pi )(Y) AZ + \pi(Z) (\overset{g}{\nabla}_{X} A )Y,
	\end{split}
\end{equation}
and $\overset{2}{\alpha}$ and $\overset{2}{\gamma}$ are given with equations (\ref{eq:alpha2}) and (\ref{eq:gamma2}).
\begin{theorem}
	The Riemannian curvature tensor $\overset{g}{R} (X,Y)Z$ is invariant under connection transformation $\overset{g}{\nabla} \rightarrow \overset{2}{\nabla}$ if and only if tensor $\overset{2}{M} (X,Y)Z$ is symmetric with respect to $X$ and $Y$, where $\overset{2}{M}$ is given with (\ref{eq:M2XYZQ-Snm}), $\overset{g}{\nabla}$ denotes Levi-Civita connection and $\overset{2}{\nabla}$ denotes the quarter-symmetric non-metric connection (\ref{eq:Q-S-nmetric2}).
\end{theorem}

The following theorem presents the skew-symmetric properties of curvature tensors.
\begin{theorem}
	Let $(\mathcal{M}, G=g+F)$ be a generalized Riemannian manifold with the quarter-symmetric non-metric connection (\ref{eq:Q-Snmetric}). Curvature tensors $\overset{\theta}{R} (X,Y)Z$, $\theta = 1, \dots, 5$ satisfy the following identities
	\begin{align*}
		\overset{\alpha}{R} (X,Y)Z  = & - \overset{\alpha}{R} (Y,X)Z, \; \alpha=1,2,
		\\
		\overset{\beta}{R} (X,Y)Z  = & - \overset{\beta}{R} (Y,X)Z + (\overset{g}{\nabla}_{Y} \pi)(Z)AX + (\overset{g}{\nabla}_{X} \pi)(Z)AY - ((\overset{g}{\nabla}_{X} \pi)(Y) + (\overset{g}{\nabla}_{Y} \pi)(X))AZ \\
		& - \pi(X) (\overset{g}{\nabla}_{Y} A)Z - \pi(Y) (\overset{g}{\nabla}_{X} A)Z + \pi(Z)( (\overset{g}{\nabla}_{X} A)Y + (\overset{g}{\nabla}_{Y} A)X), \; \beta=3,4,
		\\
		\overset{5}{R} (X,Y)Z  = & - \overset{5}{R} (Y,X)Z + \frac{1}{2} \pi(X) (\pi(Y) A^2 Z - \pi(AZ) AY) + \frac{1}{2} \pi(Y) (\pi (X) A^2 Z - \pi(AZ) AX) \\
		& - \frac{1}{2} \pi(Z) (\pi(X) A^2Y + \pi(Y) A^2X - \pi(AX)AY- \pi(AY)AX).
	\end{align*}
\end{theorem}
\begin{proof} We will prove the relation for curvature tensor $\overset{5}{R}$. If we add up equations (\ref{eq:R5XYZQ-Snm1}) and
	\begin{equation}
		\label{eq:R5YXZQ-Snm1}
		\begin{split}
			\overset{5}{R} (Y,X)Z = & \overset{g}{R} (Y,X)Z  + \frac{1}{4} \pi(Y) (\pi(X) A^2 Z - \pi(AZ) AX) + \frac{1}{4} \pi(X) (\pi (Y) A^2 Z - \pi(AZ) AY) \\
			& - \frac{1}{4} \pi(Z) (\pi(Y) A^2X + \pi(X) A^2Y - \pi(AY)AX- \pi(AX)AY), 
		\end{split}
	\end{equation}
	we get the equation
	\begin{equation*}
		\begin{split}
			\overset{5}{R} (X,Y)Z + \overset{5}{R} (Y,X)Z  = &  \frac{1}{2} \pi(X) (\pi(Y) A^2 Z - \pi(AZ) AY) + \frac{1}{2} \pi(Y) (\pi (X) A^2 Z - \pi(AZ) AX) \\
			& - \frac{1}{2} \pi(Z) (\pi(X) A^2Y + \pi(Y) A^2X - \pi(AX)AY- \pi(AY)AX),
	\end{split}\end{equation*}
	where we used the skew-symmetry of tensor $\overset{g}{R}$, i.e. $\overset{g}{R} (X,Y)Z = -\overset{g}{R} (Y,X)Z$.
\end{proof}

Based on relations (\ref{eq:R5XYZQ-Snm1}) and (\ref{eq:R5YXZQ-Snm1}), we conclude that curvature tensor $\overset{5}{R}$ has the following property
\begin{equation*}
	\overset{5}{R} (X,Y)Z - \overset{5}{R} (Y,X)Z = 2\overset{g}{R} (X,Y)Z.
\end{equation*}

If we cyclically sum equations (\ref{eq:R1XYZQ-Snm}), (\ref{eq:R2XYZQ-Snm1}), (\ref{eq:R3XYZQ-Snm1}), (\ref{eq:R4XYZQ-Snm1}), (\ref{eq:R5XYZQ-Snm1}), we will obtain the first Bianchi identities for curvature tensors $\overset{\theta}{R} (X,Y)Z$, $\theta = 1, \dots, 5$.

\begin{theorem}
	Let $(\mathcal{M}, G=g+F)$ be a generalized Riemannian manifold with the quarter-symmetric non-metric connection (\ref{eq:Q-Snmetric}). Curvature tensors $\overset{\theta}{R} (X,Y)Z$, $\theta = 1, \dots, 5$ satisfy the following identities
	\begin{align*}
		\underset{XYZ}{\sigma}	\overset{1}{R} (X,Y)Z  = & \underset{XYZ}{\sigma} \big( \pi(X) ((\overset{g}{\nabla}_{Y} A)Z - (\overset{g}{\nabla}_{Z} A)Y) - ((\overset{g}{\nabla}_{X} \pi )(Y) - (\overset{g}{\nabla}_{Y} \pi )(X)) AZ  \\ 
		& \qquad -\frac{1}{2} (\pi(X)\pi(AY) - \pi(AX)\pi(Y))AZ \big),
		\\
		\underset{XYZ}{\sigma}	\overset{2}{R} (X,Y)Z  = & \underset{XYZ}{\sigma} \big(  ((\overset{g}{\nabla}_{X} \pi )(Y) - (\overset{g}{\nabla}_{Y} \pi )(X)) AZ - \pi(X) ((\overset{g}{\nabla}_{Y} A)Z - (\overset{g}{\nabla}_{Z} A)Y)  \\ 
		& \qquad -\frac{1}{2} (\pi(X)\pi(AY) - \pi(AX)\pi(Y))AZ \big),
		\\
		\underset{XYZ}{\sigma}	\overset{3}{R} (X,Y)Z  = & \underset{XYZ}{\sigma} (\pi(X)\pi(AY) - \pi(AX)\pi(Y))AZ,
		\\
		\underset{XYZ}{\sigma} \overset{\theta}{R} (X,Y)Z  = & 0, \; \theta=4,5.
	\end{align*}
\end{theorem}

If holds $\overset{1}{R} (X,Y)Z = \overset{2}{R} (X,Y)Z$, then connection $\overset{1}{\nabla}$ is \textit{dual symmetric}.

\begin{theorem}
	Quarter-symmetric non-metric connection $\overset{1}{\nabla}$ given with (\ref{eq:Q-Snmetric}) is dual symmetric connection if and only if tensor $V(X,Y)Z$ is symmetric with respect to vectors $X$ and $Y$, where $V$ is (1,3) type tensor given with
	\begin{equation}\label{eq:VXYZQ-Snm}
		V(X,Y)Z= \pi(X)(\overset{g}{\nabla}_{Y} A )Z  + \pi(Z)(\overset{g}{\nabla}_{X} A )Y + (\overset{g}{\nabla}_{X} \pi)(Z) AY - (\overset{g}{\nabla}_{X} \pi)(Y)AZ.
	\end{equation}
\end{theorem}
\begin{proof}
	Based on equations (\ref{eq:R1XYZQ-Snm}) and (\ref{eq:R2XYZQ-Snm1}), the relation between curvature tensors $\overset{1}{R}$ and $\overset{2}{R}$ is easily obtained
	\begin{equation*}
		\begin{split}
			\overset{2}{R} (X,Y)Z = \overset{1}{R} (X,Y)Z  &- (\overset{g}{\nabla}_{X} \pi)(Z) AY + (\overset{g}{\nabla}_{Y} \pi)(Z) AX + \overset{1}{\beta}(X,Y)AZ  \\ & +\pi(Y)(\overset{g}{\nabla}_{X} A )Z - \pi(X)(\overset{g}{\nabla}_{Y} A )Z - \overset{1}{\delta}(X,Y) \pi(Z).
	\end{split}\end{equation*}
	From here we see that the connection $\overset{1}{\nabla}$ is dual symmetric if and only if the following holds
	\begin{equation*}
		\pi(X)(\overset{g}{\nabla}_{Y} A )Z - \pi(Y)(\overset{g}{\nabla}_{X} A )Z + \overset{1}{\delta}(X,Y) \pi(Z) + (\overset{g}{\nabla}_{X} \pi)(Z) AY - (\overset{g}{\nabla}_{Y} \pi)(Z) AX - \overset{1}{\beta}(X,Y)AZ=0,
	\end{equation*}
	i.e.
	\begin{equation*}
		V(X,Y)Z=V(Y,X)Z,
	\end{equation*}
	where $V$ is given with equation (\ref{eq:VXYZQ-Snm}).	
\end{proof}

\begin{remark}
	In paper \cite{dizhao2022}, the authors observed quarter-symmetric non-metric connection in the form
	\begin{equation*}
		\overset{1}{\nabla}_{X} Y = \overset{g}{\nabla}_{X} Y + \frac{1}{2}\pi(Y) \phi X - \frac{1}{2}\pi(X) \phi Y, 
	\end{equation*}
	where $\phi$ is (1,1) tensor, while in this paper we used (1,1) tensor $A$, which is associated with skew-symmetric tensor $F$. Also, the content of this paper is different than that of paper \cite{dizhao2022}.
\end{remark}

\section{Further work}

Considering that almost complex, almost para-complex, almost contact and almost para-contact manifolds are examples of a generalized Riemannian manifold, we will continue the research on this quarter-symmetric connection on the mentioned manifolds and deal with the application of the results obtained in this paper.

\section{Acknowledgement}
The financial support of this work by the project of the Ministry of Education, Science and Technological Development of the Republic of Serbia (project no. 451-03-65/2024-03/200123) and by project of Faculty of Sciences and Mathematics, University of Pri\v stina in Kosovska Mitrovica (internal-junior project IJ-2303).

\end{document}